\documentclass[letterpaper,11pt,twoside,keywordsasfootnote,addressatend,noinfoline]{article}
\usepackage{fullpage}
\usepackage[english]{babel}
\usepackage{amssymb}
\usepackage{amsmath}
\usepackage{theorem}
\usepackage{epsfig}
\usepackage{subfigure}
\usepackage{imsart}
\usepackage{color}

\theorembodyfont{\normalfont}
\newtheorem{theorem}{Theorem}

\newtheorem{lemma}[theorem]{Lemma}

\newenvironment{proof}{\noindent{\scshape Proof.}}{\hspace*{2mm}~$\square$}

\newcommand{\N}{\mathbb{N}}
\newcommand{\Z}{\mathbb{Z}}
\newcommand{\R}{\mathbb{R}}
\newcommand{\ind}{\mathbf{1}}
\newcommand{\ep}{\epsilon}

\newcommand{\blue}[1]{\textcolor{black}{#1}}

\newcommand{\best}{\eta}
\newcommand{\spar}{\bar \eta}
\newcommand{\hypc}{\zeta}
\newcommand{\rich}{\pi}
\newcommand{\boot}{\xi}

\DeclareMathOperator{\card}{card}

\begin{document}
\begin{frontmatter}
\title     {Evolutionary games on the lattice: \\ best-response dynamics}
\runtitle  {Best-response dynamics}
\author    {Stephen Evilsizor and Nicolas Lanchier\thanks{Both authors were partially supported by NSF Grant DMS-10-05282}}
\runauthor {S. Evilsizor and N. Lanchier}
\address   {School of Mathematical and Statistical Sciences \\ Arizona State University \\ Tempe, AZ 85287, USA.}

\maketitle

\begin{abstract} \ \
 The best-response dynamics is an example of an evolutionary game where players update their strategy in order to maximize
 their payoff.
 The main objective of this paper is to study a stochastic spatial version of this game based on the framework of interacting
 particle systems in which players are located on an infinite square lattice.
 In the presence of two strategies, and calling a strategy selfish or altruistic depending on a certain ordering of the
 coefficients of the underlying payoff matrix, a simple analysis of the nonspatial mean-field approximation of the spatial
 model shows that a strategy is evolutionary stable if and only if it is selfish, making the system bistable when both
 strategies are selfish.
 The spatial and nonspatial models agree when at least one strategy is altruistic.
 In contrast, we prove that, in the presence of two selfish strategies \blue{and in any spatial dimensions}, only the most
 selfish strategy remains evolutionary stable.
\blue{The main ingredients of the proof are monotonicity results and a coupling between the best-response dynamics properly
 rescaled in space with bootstrap percolation to compare the infinite time limits of both systems.}
\end{abstract}



\end{frontmatter}


\section{Introduction}
\label{sec:intro}

\indent The framework of evolutionary game theory, which describes the dynamics of populations of individuals identified to players,
 has been initiated by theoretical biologist Maynard Smith and first appeared in his work with Price \cite{maynardsmith_price_1973}.
 Each individual-player is characterized by one of a finite number~$n$ of possible strategies and is attributed a payoff that is
 calculated based on the strategy of the surrounding players and an~$n \times n$ payoff matrix.
 The most popular model of evolutionary game is probably the so-called replicator equation reviewed in \cite{hofbauer_sigmund_1998},
 a system of deterministic differential equations for the frequencies of players holding a given strategy.
 This paper is a sequel of the second author's work \cite{lanchier_2014} continuing the analytical study of evolutionary
 games based on the framework of interacting particle systems which, in contrast with the replicator equation, also includes
 stochasticity and space in the form of local interactions. \vspace*{5pt}


\noindent{\bf Model description} --
 The version of the best-response dynamics we consider in this paper is a continuous-time Markov chain whose state at time~$t$ is
 a spatial configuration
 $$ \best_t : \Z^d \longrightarrow \blue{\{1, 2 \}} := \hbox{the set of strategies}. $$
 In words, each point of the~$d$-dimensional square lattice is occupied by exactly one player who is characterized by her strategy.
 The spatial structure is included in the form of local interactions assuming that each player's payoff only depends on the
 strategy of her~$2d$ neighbors.
 More precisely, having a~\blue{two by two} payoff matrix~$A = (a_{ij})$ where~$a_{ij}$ is interpreted as the payoff of a player holding
 strategy~$i$ interacting with a player holding strategy~$j$, each configuration is turned into a so-called payoff landscape that
 attributes a payoff to each vertex as follows:
 $$ \begin{array}{rcl} \blue{\phi (x, \best_t)} & := & \blue{(a_{11} \,N_1 (x, \best_t) + a_{12} \,N_2 (x, \best_t)) \ \ind \{\best_t (x) = 1 \}} \vspace*{4pt} \\ && \hspace*{15pt} + \
                                                       \blue{(a_{21} \,N_1 (x, \best_t) + a_{22} \,N_2 (x, \best_t)) \ \ind \{\best_t (x) = 2 \}} \quad \hbox{for all} \quad x \in \Z^d \end{array} $$
 where~$N_j (x, \best_t)$ is the number of type~$j$ neighbors of vertex~$x$, i.e.,
 $$ N_j (x, \best_t) \ := \ \card \,\{y \in \Z^d : y \sim x \ \hbox{and} \ \best_t (y) = j \} $$
 where \blue{the binary relationship~$\sim$ indicates that two vertices are neighbors.}
 In the traditional framework of evolutionary game theory, each strategy is often interpreted as a trait and each payoff defined
 through the payoff landscape as a fitness or reproduction success.
 In particular, evolutionary game theory makes the implicit assumption that players are not rational decision-makers who can choose
 their strategy and that the evolution of the system is driven by births and deaths.
 In contrast, the best-response dynamics assumes that players are rational decision-makers changing their strategy in order to maximize
 their payoff.
 Specifically, we assume that each player updates her strategy at \blue{an exponential rate one} choosing to change her strategy if
 and only if it increases her payoff.
 In particular, in case of a tie, i.e., the player would not change her payoff by changing her strategy, nothing happens.
 More precisely, letting
\begin{equation}
\label{eq:model-1}
  \begin{array}{rcl}
   \phi_1 (x, \best_t) & := & a_{11} \,N_1 (x, \best_t) + a_{12} \,N_2 (x, \best_t) \quad \hbox{for all} \quad x \in \Z^d \vspace*{4pt} \\
   \phi_2 (x, \best_t) & := & a_{21} \,N_1 (x, \best_t) + a_{22} \,N_2 (x, \best_t) \quad \hbox{for all} \quad x \in \Z^d \end{array}
\end{equation}
 be the payoff that the player at~$x$ would receive if she followed strategy~1 and~2, respectively, the best-response dynamics
 is formally described by the Markov generator
\begin{equation}
\label{eq:model-2}
   \begin{array}{l} Lf (\best_t) \ = \
   \sum_x \,\ind \{\phi_1 (x, \best_t) > \phi_2 (x, \best_t) \} \ [f (\best_t^{x, 1}) - f (\best_t)] \vspace*{6pt} \\ \hspace*{70pt} + \
   \sum_x \,\ind \{\phi_1 (x, \best_t) < \phi_2 (x, \best_t) \} \ [f (\best_t^{x, 2}) - f (\best)] \end{array}
\end{equation}
 where the configuration~$\best_t^{x, i}$ is obtained from~$\best_t$ by setting to~$i$ the strategy at~$x$ and leaving the strategy at the
 other vertices unchanged.
 Note that, for any given vertex~$x$, the difference between the two alternative payoffs in \eqref{eq:model-1} can be written as 
 $$ \begin{array}{rcl}
    \phi_1 (x, \best_t) - \phi_2 (x, \best_t) & = &
      (a_{11} \,N_1 (x, \best_t) + a_{12} \,N_2 (x, \best_t)) - (a_{21} \,N_1 (x, \best_t) + a_{22} \,N_2 (x, \best_t)) \vspace*{4pt} \\ & = &
      (a_{11} - a_{21}) \,N_1 (x, \best_t) - (a_{22} - a_{12}) \,N_2 (x, \best_t). \end{array} $$
 In particular, the dynamics only depends on~$a_1 := a_{11} - a_{21}$ and~$a_2 := a_{22} - a_{12}$ rather than all four coefficients
 of the payoff matrix so the Markov generator \eqref{eq:model-2} can be written as
\begin{equation}
\label{eq:model-3}
   \begin{array}{l} Lf (\best_t) \ = \
   \sum_x \,\ind \{a_1 \,N_1 (x, \best_t) > a_2 \,N_2 (x, \best_t) \} \ [f (\best_t^{x, 1}) - f (\best_t)] \vspace*{6pt} \\ \hspace*{70pt} + \
   \sum_x \,\ind \{a_1 \,N_1 (x, \best_t) < a_2 \,N_2 (x, \best_t) \} \ [f (\best_t^{x, 2}) - f (\best_t)]. \end{array}
\end{equation}
 Since the behavior of the system strongly depends on the sign of~$a_1$ and~$a_2$, it is convenient to use the terminology introduced
 in \cite{lanchier_2013, lanchier_2014} by declaring
 strategy~$i$ to be
\begin{itemize}
 \item {\bf altruistic} when~$a_i < 0$, meaning that a player with strategy~$i$ confers a lower payoff to a player following the same strategy
  than to a player following the other strategy, \vspace*{4pt}
 \item {\bf selfish} when~$a_i > 0$, meaning that a player with strategy~$i$ confers a higher payoff to a player following the same strategy
  than to a player following the other strategy.
\end{itemize} \vspace*{5pt}


\noindent{\bf Mean-field approximation} --
 To understand the role of space in the long-term behavior of the best-response dynamics, the first step is to look at the
 deterministic nonspatial version, or mean-field approximation, of the process \eqref{eq:model-3}.
 This mean-field model is obtained under the assumption that the population is well-mixing, and more precisely by looking at the
 process on the complete graph in which any two players are neighbors and then taking the limit as the number of vertices tends
 to infinity.
 This results in a system of differential equations for the frequency of players holding strategy~$i$ that we denote by~$u_i$.
 In the absence of a spatial structure, the payoff that a player would receive if she followed strategy~1 and 2, respectively, is
 $$ \phi_1 (u_1, u_2) \ = \ a_{11} \,u_1 + a_{12} \,u_2 \qquad \hbox{and} \qquad \phi_2 (u_1, u_2) \ = \ a_{21} \,u_1 + a_{22} \,u_2 $$
 which can be viewed as the nonspatial analog of \eqref{eq:model-1}.
 Also, under the evolution rules of the best-response dynamics, either each type~1 player or each type~2 player changes her strategy
 at \blue{an exponential rate one} depending on whether~$\phi_1 - \phi_2$ is negative or positive, respectively.
 Then, rescaling time by the number of vertices and taking the limit as the number of vertices tends to infinity gives the following
 differential equation for the frequency of type~1 players:
\begin{equation}
\label{eq:meanfield-1}
 \begin{array}{rcl}
   u_1' (t) & = & u_2 \ \ind \{\phi_1 (u_1, u_2) > \phi_2 (u_1, u_2) \} \ - \ u_1 \ \ind \{\phi_1 (u_1, u_2) < \phi_2 (u_1, u_2) \} \vspace*{4pt} \\
            & = & u_2 \ \ind \{a_1 \,u_1 > a_2 \,u_2 \} \ - \ u_1 \ \ind \{a_1 \,u_1 < a_2 \,u_2 \} \vspace*{4pt} \\
            & = & u_2 \ \ind \{(a_1 + a_2) \,u_1 > a_2 \} \ - \ u_1 \ \ind \{(a_1 + a_2) \,u_1 < a_2 \} \end{array}
\end{equation}
 where we used that~$u_1 + u_2 = 1$.
\blue{Letting~$u_* := a_2 \,(a_1 + a_2)^{-1}$, we have
 $$ \begin{array}{rcccc}
      u_1' (t) \ = \ + \,u_2 & \hbox{when} & (u_1 > u_* \ \hbox{and} \ a_1 + a_2 > 0) & \hbox{or} & (u_1 < u_* \ \hbox{and} \ a_1 + a_2 < 0) \vspace*{2pt} \\
      u_1' (t) \ = \ - \,u_1 & \hbox{when} & (u_1 > u_* \ \hbox{and} \ a_1 + a_2 < 0) & \hbox{or} & (u_1 < u_* \ \hbox{and} \ a_1 + a_2 > 0) \end{array} $$
 which shows the following four possible regimes:}
\begin{itemize}
 \item when strategy~1 is selfish and strategy~2 altruistic, strategy~1 wins in the sense that
   starting from any initial condition~$u_1 (t) \to 1$ as~$t \to \infty$. \vspace*{4pt}
 \item when strategy~1 is altruistic and strategy~2 selfish, strategy~2 wins in the sense that
   starting from any initial condition~$u_1 (t) \to 0$ as~$t \to \infty$. \vspace*{4pt}
 \item when both strategies are altruistic, coexistence occurs in the sense that
   starting from any initial condition~$u_1 (t) \to \blue{u_*} \in (0, 1)$ as~$t \to \infty$. \vspace*{4pt}
 \item when both strategies are selfish, the system is bistable:
  $$ \begin{array}{rcl}
       u_1 (t) \to 0 \ \ \hbox{as} \ \ t \to \infty & \hbox{when} & u_1 (0) < \blue{u_*} \in (0, 1) \vspace*{2pt} \\
       u_1 (t) \to 1 \ \ \hbox{as} \ \ t \to \infty & \hbox{when} & u_1 (0) > \blue{u_*} \in (0, 1). \end{array} $$
\end{itemize}
 In terms of evolutionary stable strategy, this indicates that, for well-mixing populations, a strategy is evolutionary stable if it
 is selfish but not if it is altruistic.
\blue{Recall that an evolutionary stable strategy is defined as a strategy which, if adopted by a population, cannot be invaded by any
 alternative strategy starting at an infinitesimally small frequency.} \vspace*{5pt}


\noindent{\bf Spatial stochastic model} --
 We now return to the spatial model \eqref{eq:model-3} looking at the four parameter regions corresponding to the four possible regimes
 of the mean-field approximation.
 Assuming first that strategy~1 is selfish and strategy~2 altruistic, we get
 $$ a_1 \,N_1 (x, \best_t) - a_2 \,N_2 (x, \best_t) \ = \ a_1 \,N_1 (x, \best_t) + (- a_2)(2d - N_1 (x, \best_t)) \ > \ 0 $$
 for all~$x \in \Z^d$ and all configuration~$\best_t$.
 This shows that each type~2 player changes her strategy at \blue{an exponential rate one} whereas each type~1 player sticks to her
 strategy, therefore strategy~1 wins, just as in the mean-field model, in the sense that for any initial configuration
 $$ \begin{array}{l} \lim_{t \to \infty} \,P \,(\best_t (x) = 1) \ = \ 1 \quad \hbox{for all} \quad x \in \Z^d. \end{array} $$
 By symmetry, strategy~2 wins whenever strategy~1 is altruistic and strategy~2 selfish.
 Note in particular that the ``all 1'' and ``all 2'' configurations are not necessarily absorbing states for the process.
 This is due to the fact that, though the new strategy is chosen based on the strategy of the neighbors, it is not chosen from
 the neighborhood.
 Looking now at altruistic-altruistic interactions, whenever the player at~$x$ and all her neighbors follow the same strategy,
 $$ \begin{array}{rcl}
     a_1 \,N_1 (x, \best_t) - a_2 \,N_2 (x, \best_t) \ = \ + \ 2d \,a_1 \ < \ 0 & \hbox{when} & \best_t (x) = 1 \vspace*{2pt} \\
     a_1 \,N_1 (x, \best_t) - a_2 \,N_2 (x, \best_t) \ = \ - \ 2d \,a_2 \ > \ 0 & \hbox{when} & \best_t (x) = 2. \end{array} $$
 In either case, the player at~$x$ changes her strategy at \blue{an exponential rate one}, indicating that, as in the mean-field model,
 two altruistic strategies coexist in the sense that
 $$ \begin{array}{l} \lim_{t \to \infty} \,P \,(\best_t (x) = \best_t (y)) \ < \ 1 \quad \hbox{for all} \quad x, y \in \Z^d, \ x \neq y. \end{array} $$
 We now study the process when both strategies are selfish, a case more challenging mathematically and also more interesting
 as it shows some important disagreements between the spatial and nonspatial models.
 To confront our results for the spatial model with the bistability displayed by its nonspatial counterpart, we consider the
 process starting from the product measure with
 $$ P \,(\best_0 (x) = 1) \ =: \ p \quad \hbox{for all} \quad x \in \Z^d $$
 and compare the models when~$p = u_1 (0)$.
\blue{The fact that the inclusion of space in the form of local interactions strongly affects the long-term behavior of the system
 can be seen in a specific parameter region using a standard coupling with the Richardson model~\cite{richardson_1973}.}
 Indeed, let
\begin{equation}
\label{eq:change}
  \begin{array}{l} c (x, \best_t) \ := \ \blue{\lim_{h \to 0} \,P \,(\best_{t + h} (x) \neq \best_t (x) \,| \,\best_t)}. \end{array}
\end{equation}
 Then, when~$a_1 > (2d - 1) \,a_2 > 0$ and~$N_1 (x, \best_t) \geq 1$, we have
 $$ \begin{array}{rcl}
     c (x, \best_t \,| \,\best_t (x) = 1) & = & \ind \,\{a_1 \,N_1 (x, \best_t) < a_2 \,N_2 (x, \best_t) \} \ \leq \ \ind \,\{a_1 < (2d - 1) \,a_2 \} \ = \ 0 \vspace*{4pt} \\
     c (x, \best_t \,| \,\best_t (x) = 2) & = & \ind \,\{a_1 \,N_1 (x, \best_t) > a_2 \,N_2 (x, \best_t) \} \ \geq \ \ind \,\{a_1 > (2d - 1) \,a_2 \} \ = \ 1 \end{array} $$
 almost surely.
 These two inequalities imply that the set of type~1 players dominates \blue{stochastically} the set of infected sites in the Richardson
 model~$\rich_t$ with initial configuration
 $$ \blue{\rich_0 (x) \ = \ \ind \,\{\best_0 (x) = 1 \ \hbox{and} \ \best_0 (y) = 1 \ \hbox{for some} \ y \sim x \}} $$
\blue{which, in turns, implies that strategy~1 wins whenever~$p > 0$.
 This shows in particular the existence of parameter regions in which, in contrast with the nonspatial model, only the most selfish
 strategy is evolutionary stable for the spatial model.
 Returning to general selfish-selfish interactions, the numerical simulations of the two-dimensional process displayed in
 Figure~\ref{fig:simulations} suggest that, when~$a_1$ is slightly larger than~$a_2$ and the initial density~$p > 0$ is small, the
 system fixates to a configuration in which the set of type~1 players consists of a union of disjoint rectangles, indicating that strategy~1
 is unable to invade strategy~2.
 These simulations, however, are misleading due to the finiteness of the graph, and it can be proved that, in any dimensions, the most selfish
 strategy always wins even when starting at a low density.
 More precisely, we have the following theorem.
\begin{figure}[t]
\centering
\mbox{\subfigure[$a_1 = 1.01 > a_2 = 1$ and~$p = 0.15$]{\epsfig{figure=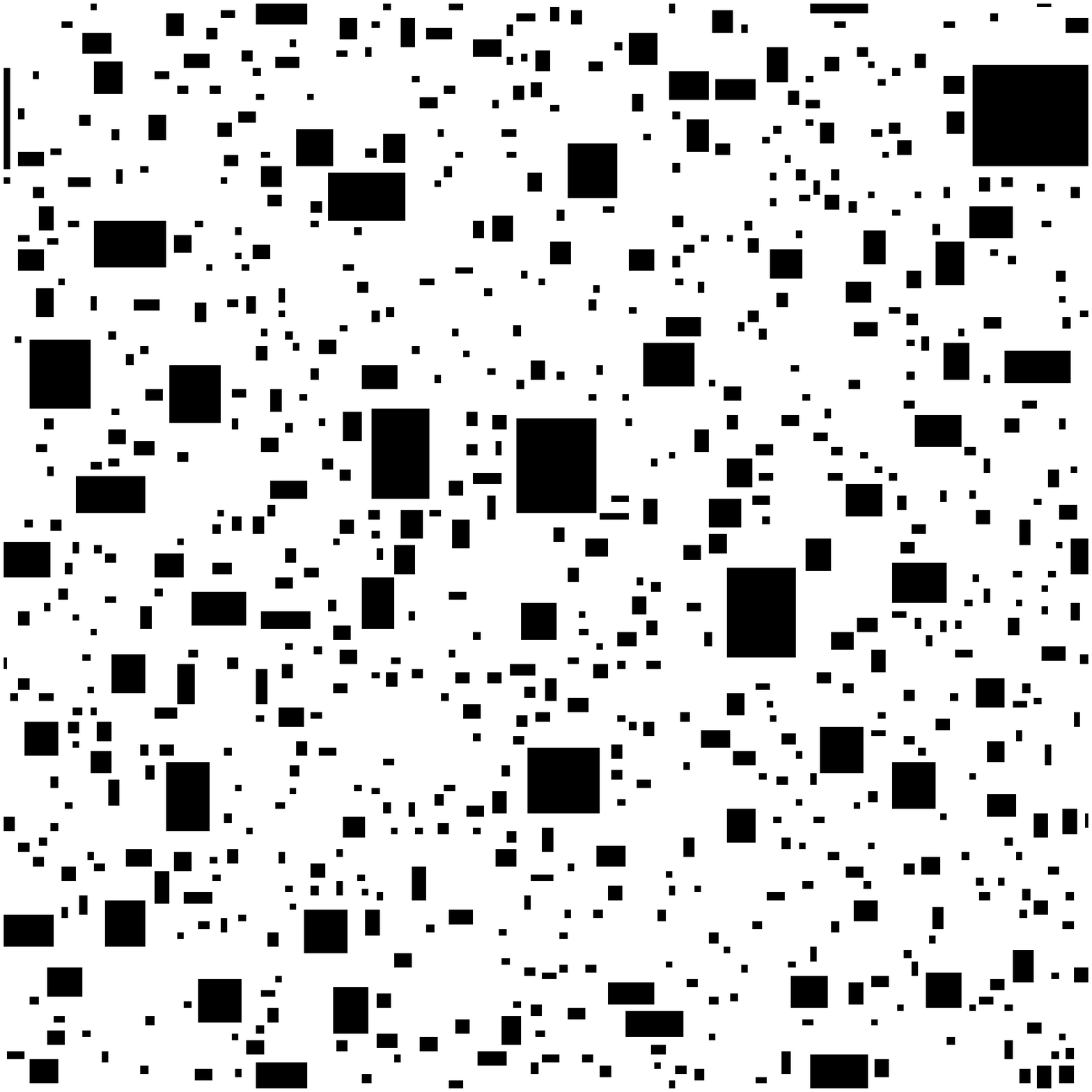, width=0.45\textwidth}} \hspace*{5pt}
      \subfigure[$a_1 = 1.01 > a_2 = 1$ and~$p = 0.20$]{\epsfig{figure=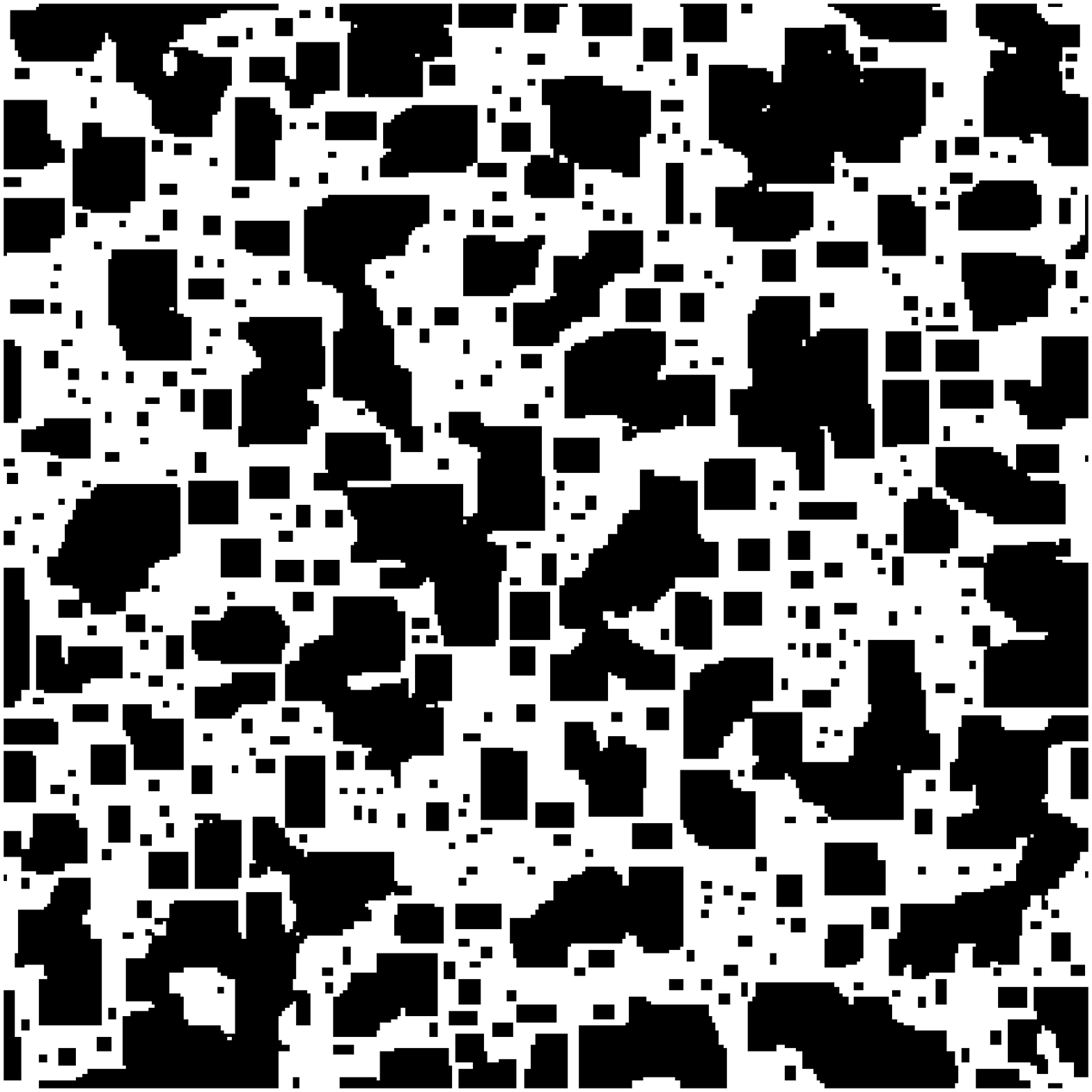, width=0.45\textwidth}}}
\caption{\upshape{Best-response dynamics on a~$300 \times 300$ lattice with periodic boundary conditions starting from a product
 measure with density~$p$ of type 1 players in black.
 On the left picture, the process hits an absorbing state in which both types are present, whereas on the right picture,
 which shows a snapshot of the process at time 25, the system is converging to the all black configuration: strategy~1 wins.}}
\label{fig:simulations}
\end{figure}
\begin{theorem} --
\label{th:selfish}
 Assume that~$a_1 > a_2 > 0$ and~$p > 0$. Then,
 $$ \begin{array}{l} \lim_{t \to \infty} \,P \,(\best_t (x) = 1) \ = \ 1 \quad \hbox{for all} \quad x \in \Z^d. \end{array} $$
\end{theorem}
 In particular, while any selfish strategy is evolutionary stable in the nonspatial model, only the most selfish strategy is evolutionary
 stable in the spatial model.
 The result in one dimension directly follows from our coupling with the Richardson model since
 $$ (2d - 1) \,a_2 \ = \ a_2 \quad \hbox{when} \quad d = 1 $$
 while the general result relies on a combination of monotonicity results and coupling arguments to compare the best-response dynamics
 with bootstrap percolation.
 More precisely, we first prove that, in the presence of selfish-selfish interactions, the best-response dynamics is attractive, which
 allows to focus on the process starting from a certain reduced configuration that consists of a union of hyperrectangles.
 The second ingredient is to show that, for the process starting from this reduced configuration, the set of type~1 players is a pure growth process, just
 like the Richardson model.
 This strong monotonicity result is then applied repeatedly to show that the best-response dynamics properly rescaled in space dominates
 stochastically bootstrap percolation with parameter~$d$.
 From this domination and a result due to Schonmann~\cite[Theorem~3.1]{schonmann_1992}, we finally deduce that, unlike
 what Figure~\ref{fig:simulations} suggests, the most selfish strategy indeed invades the entire lattice.}


\section{Some monotonicity results}
\label{sec:monotonicity}

\indent To avoid cumbersome notations, it is convenient to sometimes think of the state of the process as a subset rather than a
 function by using the identification:
 $$ \blue{\best_t \ \equiv \ \{x \in \Z^d : \best_t (x) = 1 \} \ \subset \ \Z^d.} $$
 One key ingredient is to think of the process as being constructed from a so-called Harris' graphical representation~\cite{harris_1972}
 which, in the case of the best-response dynamics, reduces to a collection of independent Poisson processes.
 More precisely,
\begin{itemize}
 \item for each~$x \in \Z^d$, we let~$(N_t (x) : t \geq 0)$ be a rate one Poisson process and \vspace*{4pt}
 \item we denote by~$T_n (x) := \inf \,\{t : N_t (x) = n \}$ its~$n$th arrival time.
\end{itemize}
 The configuration at time~$t := T_n (x)$ is obtained from~$\best_{t-}$ by
 $$ \begin{array}{rcl}
      \hbox{adding} \ \ x & \hbox{when} & a_1 \,N_1 (x, \best_{t-}) > a_2 \,N_2 (x, \best_{t-}) \vspace*{2pt} \\
    \hbox{removing} \ \ x & \hbox{when} & a_1 \,N_1 (x, \best_{t-}) < a_2 \,N_2 (x, \best_{t-}). \end{array} $$
 An argument due to Harris \cite{harris_1972} implies that the best-response dynamics starting from any initial configuration can indeed
 be constructed using this rule.
 The next lemma shows that, in the presence of selfish-selfish interactions, the best-response dynamics is attractive.
\begin{lemma} --
\label{lem:attractive}
 The process with~$a_1 > 0$ and~$a_2 > 0$ is attractive:
 $$ P \,(x \in \spar_t) \ \leq \ P \,(x \in \best_t) \quad \hbox{whenever} \quad \spar_0 \subset \best_0. $$
\end{lemma}
\begin{proof}
 Let~$\spar_t \subset \best_t$.
 Since~$a_1 > 0$ and~$a_2 > 0$,
\begin{equation}
\label{eq:attractive-1}
   a_1 \,N_1 (x, \spar_t) \ \leq \ a_1 \,N_1 (x, \best_t) \quad \hbox{and} \quad a_2 \,N_2 (x, \spar_t) \ \geq \ a_2 \,N_2 (x, \best_t).
\end{equation}
 Let~$c (x, \best_t)$ be defined as in~\eqref{eq:change}.
 Using \eqref{eq:attractive-1}, we obtain that, for all~$x \in \spar_t$,
\begin{equation}
\label{eq:attractive-2}
 \begin{array}{rcl}
   c (x, \spar_t) & = &    \ind \,\{a_1 \,N_1 (x, \spar_t) < a_2 \,N_2 (x, \spar_t) \} \vspace*{4pt} \\
                       & \geq & \ind \,\{a_1 \,N_1 (x, \best_t) < a_2 \,N_2 (x, \best_t) \} \ = \ c (x, \best_t). \end{array}
\end{equation}
 Similarly, for all~$x \notin \best_t$, we have
\begin{equation}
\label{eq:attractive-3}
 \begin{array}{rcl}
   c (x, \spar_t) & = &    \ind \,\{a_1 \,N_1 (x, \spar_t) > a_2 \,N_2 (x, \spar_t) \} \vspace*{4pt} \\
                       & \leq & \ind \,\{a_1 \,N_1 (x, \best_t) > a_2 \,N_2 (x, \best_t) \} \ = \ c (x, \best_t). \end{array}
\end{equation}
 The inequalities \eqref{eq:attractive-2}--\eqref{eq:attractive-3} show that condition (B14) in Liggett \cite{liggett_1999} are satisfied, which
 proves that, in the presence of selfish-selfish interactions, the process is attractive.
\end{proof} \\ \\
 In addition to attractiveness, a key ingredient to prove our theorem is to replace the initial configuration~$\best_0$ with a specific reduced
 initial configuration~$\spar_0$.
 To define this new initial configuration, we introduce the following collection of hypercubes:
 $$ H_z \ := \ 2z + \{0, 1 \}^d \quad \hbox{for all} \quad z \in \Z^d. $$
 Then, given~$\best_0$, we say that~$H_z$ is a type~1 hypercube whenever~$H_z \subset \best_0$ and define
\begin{equation}
\label{eq:reduce}
  \begin{array}{rrl}
   \spar_0 & := & \{x \in \Z^d : x \in H_z \ \hbox{and} \ H_z \subset \best_0 \ \hbox{for some} \ z \in \Z^d \} \vspace*{4pt} \\
                &  = & \ \hbox{the union of all type~1 hypercubes}. \end{array}
\end{equation}
 Note that~$\spar_0 \subset \best_0$ therefore, according to Lemma \ref{lem:attractive},
 $$ P \,(x \in \spar_t) \ \leq \ P \,(x \in \best_t) \quad \hbox{for all} \quad (x, t) \in \Z^d \times \R_+. $$
 In particular, it suffices to prove the theorem for the modified process~$\spar_t$ \blue{that we call from now on
 the {\bf sparse} best-response dynamics}.
 The main reason for working with this process appears in the next lemma which states that, starting from any configuration that consists
 of a union of hypercubes, the process can only increase.
 This somewhat strong result is due in part to the fact that, while the time of the updates are random, the outcome at each update
 is deterministic.
\begin{lemma} --
\label{lem:expand}
 Assume that~$a_1 > a_2 > 0$.
\blue{Then,~$P \,(\spar_s \subset \spar_t \ \hbox{for all} \ s < t) = 1$.}
\end{lemma}
\begin{proof}
 Let~$\Phi$ be the function defined on the set of configurations by
\begin{equation}
\label{eq:expand-1}
  \begin{array}{rcl}
    \Phi (\best_t) & := & \{x \in \Z^d : a_1 \,N_1 (x, \best_t) > a_2 \,N_2 (x, \best_t) \vspace*{4pt} \\ && \hspace*{20pt}
                          \hbox{or} \ (x \in \best_t \ \hbox{and} \ a_1 \,N_1 (x, \best_t) = a_2 \,N_2 (x, \best_t)) \}. \end{array}
\end{equation}
 In words, while~$\best_t$ represents the set of vertices following strategy~1, configuration~$\Phi (\best_t)$ can be seen as the
 set of vertices that will become or stay of type~1 at the next update provided the configuration in their neighborhood does not
 change by the time of the update.
 Note that, due to the presence of selfish-selfish interactions:~$a_1 > 0$ and~$a_2 > 0$, we have
\begin{equation}
\label{eq:expand-2}
  \begin{array}{rcl}
   \best_t \,\subset \,\best_t' & \hbox{implies that} & N_1 (x, \best_t) \leq N_1 (x, \best_t') \ \hbox{and} \ N_2 (x, \best_t) \geq N_2 (x, \best_t') \vspace*{4pt} \\
                                & \hbox{implies that} & a_1 \,N_1 (x, \best_t) - a_2 \,N_2 (x, \best_t) \leq a_1 \,N_1 (x, \best_t') - a_2 \,N_2 (x, \best_t') \vspace*{4pt} \\
                                & \hbox{implies that} & \Phi (\best_t) \,\subset \,\Phi (\best_t') \end{array}
\end{equation}
 indicating that the function~$\Phi$ is nondecreasing.
 In addition, for any configuration~$\spar_0$ obtained by reduction of an arbitrary initial configuration using the partition into
 hypercubes, since each type~1 player has at least~$d$~type~1 neighbors and~$a_1 > a_2 > 0$, we also have
\begin{equation}
\label{eq:expand-3}
  \begin{array}{rcl}
   x \in \spar_0 & \hbox{implies that} & N_1 (x, \spar_0) \geq d \ \hbox{and} \ N_2 (x, \spar_0) \leq d \vspace*{4pt} \\
                     & \hbox{implies that} & a_1 \,N_1 (x, \spar_0) > a_2 \,N_2 (x, \spar_0) \vspace*{4pt} \\
                     & \hbox{implies that} & x \in \Phi (\spar_0) \end{array}
\end{equation}
 indicating that~$\spar_0 \subset \Phi (\spar_0)$.
 Monotonicity \eqref{eq:expand-2} and the generalization of \eqref{eq:expand-3} to all times are the main two ingredients to establish
 the lemma that we prove by induction.
 Since the lattice is infinite, the time of the first update does not exist.
 Also, in order to prove the result inductively, the next step is to use an idea of Harris \cite{harris_1972} to break down the lattice
 into finite islands that do not interact with each other for a short time.
 More precisely, we do the following construction:
\begin{itemize}
 \item we let~$\ep > 0$ be small and, for each vertex~$x$ such that~$T_1 (x) < \ep$, draw a line segment between~$x$ and each of its~$2d$ nearest neighbors.
\end{itemize}
 This construction naturally induces a partition of the lattice into clusters, where two vertices belong to the same cluster if there
 is a sequence of line segments connecting them.
 In addition, since the probability of two neighbors~$x \sim y$ being connected by a line segment
 $$ \begin{array}{l}
      P \,(\hbox{there is a line segment between~$x$ and~$y$}) \vspace*{4pt} \\ \hspace*{40pt} = \ P \,(\min (T_1 (x), T_1 (y)) < \ep) \ = \ 1 - e^{- 2 \ep} \end{array} $$
 can be made arbitrarily small by choosing time~$\ep > 0$ small, Theorem 1.33 in \cite{grimmett_1989} implies that there exists~$\ep > 0$
 small, fixed from now on, such that each cluster is almost surely finite.
 Letting~$A$ be an arbitrary, necessarily finite, cluster, we have the following two properties:
\begin{enumerate}
 \item[(a)] the configuration in~$A$ at time~$\ep$ only depends on the initial configuration of the process and its graphical representation restricted to the cluster~$A$. \vspace*{4pt}
 \item[(b)] whenever ($x \in A$ and~$N_x \not \subset A$) or ($x \in A^c$ and~$N_x \not \subset A^c$) \blue{where~$N_x$ refers to the interaction neighborhood of vertex~$x$},
  the strategy at~$x$ is not updated before time~$\ep$.
\end{enumerate}
 Now, since~$A$ is finite, the number of updates in~$A$ up to time~$\ep$ is almost surely finite and therefore can be ordered.
 Let the times of these updates and their corresponding locations be
 $$ s_0 := 0 < s_1 < s_2 < \cdots < s_m < \ep \quad \hbox{and} \quad x_1, x_2, \ldots, x_m \in A. $$
 By (a) and the definition of the function~$\Phi$, we have
 $$ x_1 \in \spar_{s_1} \quad \hbox{if and only if} \quad x_1 \in \Phi (\spar_0). $$
 But according to \eqref{eq:expand-3}, we also have~$\spar_0 \subset \Phi (\spar_0)$ therefore
\begin{equation}
\label{eq:expand-4}
   (x_1 \in \spar_{s_0} \ \ \hbox{implies} \ \ x_1 \in \spar_{s_1}) \quad \hbox{so} \quad
   (\spar_{s_0} \cap A) \,\subset \,(\spar_{s_1} \cap A) \,\subset \,(\Phi (\spar_{s_0}) \cap A).
\end{equation}
 This, together with (b) and the monotonicity of~$\Phi$ in \eqref{eq:expand-2}, implies
\begin{equation}
\label{eq:expand-5}
   (\Phi (\spar_{s_0}) \cap A) \,\subset \,(\Phi (\spar_{s_1}) \cap A) \quad \hbox{and} \quad
         (\spar_{s_1} \cap A) \,\subset \,(\Phi (\spar_{s_1}) \cap A).
\end{equation}
 The last inclusion in \eqref{eq:expand-5} allows us to repeat the same reasoning to get \eqref{eq:expand-4}--\eqref{eq:expand-5} at the next
 update time, and so on up to time~$s_m$.
 Using in addition the obvious fact that the configuration in the cluster~$A$ does not change between two consecutive updates implies that the
 property to be proved holds at all times smaller than~$\ep$ so we have
\begin{equation}
\label{eq:expand-6}
   (\spar_s \cap A) \,\subset \,(\spar_t \cap A) \quad \hbox{and} \quad (\spar_t \cap A) \,\subset \,(\Phi (\spar_t) \cap A) \quad \hbox{for all} \quad s < t \leq \ep.
\end{equation}
 This only proves the result for the process restricted to~$A$ and up to time~$\ep$.
 To extend the result across the lattice and for all times, we first use that the set of all the clusters forms a partition of
 the lattice and sum \eqref{eq:expand-6} over all the possible clusters:
\begin{equation}
\label{eq:expand-7}
  \begin{array}{l}
    \spar_s \ = \ \bigcup_A \,(\spar_s \cap A) \ \subset \ \bigcup_A \,(\spar_t \cap A) \ = \ \spar_t \quad \hbox{for all} \quad s < t \leq \ep \vspace*{4pt} \\
    \spar_{\ep} \ = \ \bigcup_A \,(\spar_{\ep} \cap A) \ \subset \ \bigcup_A \,(\Phi (\spar_{\ep}) \cap A) \ = \ \Phi (\spar_{\ep}). \end{array} 
\end{equation}
 This first inclusion proves the lemma up to time~$\ep$ while the second inclusion can be used, together with the fact that the
 process is Markov, to restart the argument and extend the result inductively up to time~$2 \ep$, then~$3 \ep$, and so on.
 This proves the result at all times.
\end{proof}


\section{Coupling with bootstrap percolation}
\label{sec:bootstrap}

\indent \blue{This section is devoted to the proof of the theorem, which relies on a coupling between bootstrap percolation and the
 best-response dynamics.
 Bootstrap percolation with parameter~$m$ is the discrete-time process whose state at time~$t$ is a spatial configuration
 $$ \boot_t : \Z^d \longrightarrow \{0, 1 \} \quad \hbox{where} \quad 0 = \hbox{empty} \quad \hbox{and} \quad 1 = \hbox{occupied} $$
 that evolves deterministically as follows: for all~$z \in \Z^d$~and~$t \in \N$,
 $$ \begin{array}{rcl}
    \boot_t (z) = 1 & \hbox{implies that} & \boot_{t + 1} (z) = 1 \vspace*{4pt} \\
    \boot_t (z) = 0 & \hbox{implies that} & \boot_{t + 1} (z) = 1 \ \ \hbox{if and only if} \ \ \card \,\{w \sim z : \boot_t (w) = 1 \} \,\geq \,m. \end{array} $$
 In view of Lemma~\ref{lem:expand} for the sparse best-response dynamics and the evolution rules of bootstrap percolation,
 both processes are almost surely monotone, therefore the limits
 $$ \begin{array}{l} \spar_{\infty} \ := \ \lim_{t \to \infty} \,\spar_t \qquad \hbox{and} \qquad \boot_{\infty} \ := \ \lim_{t \to \infty} \,\boot_t \qquad \hbox{exist}. \end{array} $$
 Here, we again identify configurations with the set of vertices in state~1.
 From now on, we call the two limit sets above, the {\bf infinite time limits} of the sparse best-response dynamics and bootstrap
 percolation, respectively.
 To prove the theorem, we first rely on the monotonicity results of the previous section to show that the infinite time limit of
 the sparse best-response dynamics properly rescaled in space dominates its counterpart for bootstrap percolation.
 The main ingredient is to couple both systems using the key function introduced in~\eqref{eq:expand-1}.
 Based on this coupling, we can directly deduce the theorem from its analog for bootstrap percolation on the infinite lattice starting
 from a product measure, a result due to Schonmann~\cite[Theorem~3.1]{schonmann_1992}.
\begin{lemma} --
\label{lem:limit}
 Assume that~$a_1 > a_2 > 0$. Then,
 $$ \Phi^n (\spar_s) \subset \spar_{\infty} \quad \hbox{almost surely for all} \quad s > 0 \quad \hbox{and} \quad n \geq 0. $$
\end{lemma}
\begin{proof}
 We prove the result by induction with respect to~$n$. \vspace*{4pt} \\
{\bf Base case} --
 This follows from~Lemma~\ref{lem:expand} which gives
 $$ P \,(\Phi^0 (\spar_s) \subset \spar_{\infty}) \ = \ P \,(\spar_s \subset \spar_{\infty}) \ \geq \ P \,(\spar_s \subset \spar_t \ \hbox{for all} \ s < t) \ = \ 1. $$
{\bf Inductive step} -- Assume~$\Phi^n (\spar_s) \subset \spar_{\infty}$ and~$x \in \Phi^{n + 1} (\spar_s) \setminus \Phi^n (\spar_s)$. Then,
\begin{equation}
\label{eq:limit-1}
  \begin{array}{rrl}
      T_y & := & \inf \,\{T > 0 : y \in \spar_T \} < \infty \ \ \hbox{a.s.} \ \ \hbox{for all} \ \ y \in \Phi^n (\spar_s) \ \ \hbox{and} \vspace*{4pt} \\
    \tau_x & := & \max \,\{T_y : y \sim x \ \hbox{and} \ y \in \Phi^n (\spar_s) \} < \infty \ \ \hbox{a.s.} \end{array}
\end{equation}
 In addition, the choice of~$x$ implies that
\begin{equation}
\label{eq:limit-2}
  a_1 \,N_1 (x, \Phi^n (\spar_s)) \ > \ a_2 \,N_2 (x, \Phi^n (\spar_s)) \quad \hbox{since} \quad x \in \Phi \,(\Phi^n (\spar_s)) \setminus \Phi^n (\spar_s)
\end{equation}
 while a new application of Lemma~\ref{lem:expand} gives
\begin{equation}
\label{eq:limit-3}
  N_1 (x, \spar_t) \ \geq \ N_1 (x, \Phi^n (\spar_s)) \quad \hbox{and} \quad N_2 (x, \spar_t) \ \leq \ N_2 (x, \Phi^n (\spar_s))
\end{equation}
 for all~$t > \tau_x$.
 Combining~\eqref{eq:limit-2}--\eqref{eq:limit-3} and using that~$a_1 > 0$ and~$a_2 > 0$, we get
 $$ \begin{array}{rcl}
      a_1 \,N_1 (x, \spar_t) & \geq & a_1 \,N_1 (x, \Phi^n (\spar_s)) \vspace*{4pt} \\
                                     & > & a_2 \,N_2 (x, \Phi^n (\spar_s)) \ \geq \ a_2 \,N_2 (x, \spar_t) \quad \hbox{for all} \quad t > \tau_x. \end{array} $$
 It follows that, given that the player at vertex~$x$ follows strategy~2 after time~$\tau_x$, she switches to strategy~1 at rate one.
 This together with~\eqref{eq:limit-1} implies that
\begin{equation}
\label{eq:limit-4}
  T_x \ = \ \inf \,\{t > 0 : x \in \spar_t \} < \infty \ \ \hbox{a.s.} \quad \hbox{therefore} \quad x \in \spar_{\infty}.
\end{equation}
 Finally, using consecutively~\eqref{eq:expand-2} and~\eqref{eq:expand-7} and then~\eqref{eq:limit-4}, we deduce that
 $$ \begin{array}{l}
    \Phi^n (\spar_s) \ \subset \ \Phi \,(\Phi^n (\spar_s)) \ = \ \Phi^{n + 1} (\spar_s) \vspace*{4pt} \\ \hspace*{25pt} \hbox{and} \quad
    \Phi^{n + 1} (\spar_s) \ = \ (\Phi^{n + 1} (\spar_s) \setminus \Phi^n (\spar_s)) \,\cup \,\Phi^n (\spar_s) \ \subset \ \spar_{\infty} \end{array} $$
 which shows the result at step~$n + 1$ and completes the proof.
\end{proof} \\ \\
 We are now ready to prove that the infinite time limit of the best-response dynamics properly rescaled in space dominates the infinite
 time limit of bootstrap percolation.
 More precisely, we look at the best-response dynamics viewed at the hypercube level by introducing
\begin{equation}
\label{eq:hypercubic}
 \hypc_t : \Z^d \longrightarrow \{0, 1 \} \quad \hbox{where} \quad \hypc_t (z) \ := \ \ind \{H_z \subset \spar_t \} \quad \hbox{for all} \quad z \in \Z^d. 
\end{equation}
 From now on, we call this process the {\bf hypercubic} best-response dynamics.
 Identifying once more configurations with the set of vertices in state~1 and using again the monotonicity of the sparse best-response dynamics
 given by~Lemma~\ref{lem:expand}, we note that
 $$ \begin{array}{rcl}
    \hypc_{\infty} \ := \ \lim_{t \to \infty} \,\hypc_t & = & \lim_{t \to \infty} \,\{z : H_z \subset \spar_t \} \vspace*{4pt} \\
                                                        & = & \{z : H_z \subset \lim_{t \to \infty} \,\spar_t \} \ = \ \{z : H_z \subset \spar_{\infty} \} \end{array} $$
 therefore the infinite time limit~$\hypc_{\infty}$ is well-defined.
\begin{lemma} --
\label{lem:coupling}
 Assume that~$a_1 > a_2 > 0$ and~$m = d$. Then,
 $$ \begin{array}{l} \boot_{\infty} \subset \hypc_{\infty} \quad \hbox{almost surely whenever} \quad \boot_0 = \hypc_0. \end{array} $$
\end{lemma}
\begin{proof}
 Let~$z \in \Z^d$ and~$s > 0$, and assume that
\begin{equation}
\label{eq:coupling-1}
  \hypc_s (z) = 0 \quad \hbox{and} \quad \card \,\{w \sim z : \hypc_s (w) = 1 \} \,\geq \,m = d.
\end{equation}
 Recalling~\eqref{eq:hypercubic}, this indicates that there are at least~$m = d$ hypercubes adjacent to~$H_z$ that are completely occupied by
 players of type~1.
 Invoking the invariance by symmetry of the best-response dynamics, we may assume without loss of generality that
\begin{equation}
\label{eq:coupling-2}
   H_{z - e_j} \subset \spar_s \quad \hbox{for} \quad j = 1, 2, \ldots, d \quad \hbox{where} \quad e_j := \hbox{$j$th unit vector}.
\end{equation}
 Since~$a_1 > a_2 > 0$, we also have
\begin{equation}
\label{eq:coupling-3}
  \Phi (\spar_s) \ \supset \ \{x \in \Z^d : N_1 (x, \spar_s) \geq N_2 (x, \spar_s) \} \ = \ \{x \in \Z^d : N_1 (x, \spar_s) \geq d \}.
\end{equation}
 Combining~\eqref{eq:coupling-2}--\eqref{eq:coupling-3} together with~Lemma~\ref{lem:expand} and some basic geometry, we get
\begin{equation}
\label{eq:coupling-4}
  \begin{array}{l} 2z + \{x \in \{0, 1 \}^d : \sum_{j = 1, 2, \ldots, d} \,x_j < n \} \ \subset \ \Phi^n (\spar_s) \quad \hbox{for} \quad n = 1, 2, \ldots, d + 1. \end{array} 
\end{equation}
 For an illustration in three dimensions, we refer to Figure~\ref{fig:bootstrap} where configuration~$\spar_s$ consists of the union of
 three hypercubes.
 In particular, taking~$n = d + 1$ gives
 $$ \begin{array}{l} H_z \ = \ 2z + \{x \in \{0, 1 \}^d : \sum_{j = 1, 2, \ldots, d} \,x_j \leq d \} \ \subset \ \Phi^{d + 1} (\spar_s). \end{array} $$
 Applying~Lemma~\ref{lem:limit}, we then obtain
\begin{equation}
\label{eq:coupling-5}
  H_z \ \subset \ \Phi^{d + 1} (\spar_s) \ \subset \ \spar_{\infty} \quad \hbox{therefore} \quad \hypc_t (z) = 1 \ \ \hbox{for some time~$t < \infty$ a.s.}
\end{equation}
 In addition, since the hypercubic process clearly inherits the monotonicity property of the sparse best-response dynamics given by~Lemma~\ref{lem:expand},
\begin{equation}
\label{eq:coupling-6}
 \begin{array}{rcl} \hypc_s (z) = 1 & \hbox{implies that} & \hypc_t (z) = 1 \quad \hbox{for all} \quad t > s. \end{array} 
\end{equation}
 In summary, \eqref{eq:coupling-6} and the fact that~\eqref{eq:coupling-1} implies~\eqref{eq:coupling-5} indicate that:
 for the hypercubic process, once a vertex is occupied it remains occupied forever, and if an empty vertex has at least~$d$ occupied neighbors then
 it becomes occupied after an almost surely finite time.
 Recalling the evolution rules of bootstrap percolation with parameter~$m = d$, the result follows.
\begin{figure}[t]
 \centering
 \includegraphics[width=0.85\textwidth]{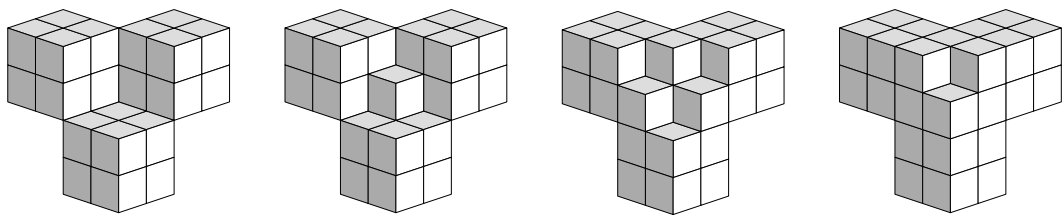}
 \caption{\upshape{Picture of the progression~$\spar_s \to \Phi (\spar_s) \to \Phi^2 (\spar_s) \to \Phi^3 (\spar_s)$ starting from the union of three
  hypercubes adjacent to the same hypercube.
  The figure gives an illustration of the inclusions in~\eqref{eq:coupling-4}.}}
\label{fig:bootstrap}
\end{figure}
\end{proof} \\ \\
 Combining the previous lemma with a result of Schonmann~\cite[Theorem 3.1]{schonmann_1992} on bootstrap percolation on the infinite lattice, we now deduce the theorem.
\begin{lemma} --
\label{lem:selfish}
 Assume that~$a_1 > a_2 > 0$ and~$p > 0$. Then,
 $$ \begin{array}{l} \lim_{t \to \infty} \,P \,(\best_t (x) = 1) \ = \ 1 \quad \hbox{for all} \quad x \in \Z^d. \end{array} $$
\end{lemma}
\begin{proof}
 To begin with, we consider bootstrap percolation with parameter~$m$ starting from the product measure with density~$q$.
 That is, the initial configuration satisfies
 $$ P \,(\boot_0 (z_1) = \boot_0 (z_2) = \cdots = \boot_0 (z_n) = 1) \ = \ q^n \quad \hbox{for} \quad z_1, z_2, \ldots, z_n \in \Z^d \quad \hbox{all distinct}. $$
 Whether the set of occupied vertices ultimately covers the entire lattice depends on the initial density and the fact that bootstrap
 percolation is clearly attractive motivates the introduction of the following critical value for the initial density:
 $$ \begin{array}{l} q_c \ := \ \inf \,\{q \in [0, 1] : P \,(\boot_{\infty} = \Z^d) = 1 \}. \end{array} $$
 Schonmann~\cite[Theorem 3.1]{schonmann_1992} proved that, when~$m \leq d$ we have~$q_c = 0$, therefore
 $$ \begin{array}{l} P \,(\boot_{\infty} = \Z^d) = 1 \quad \hbox{whenever} \quad m = d \quad \hbox{and} \quad q > 0. \end{array} $$
 In particular, taking~$q := p^{2^d}$ so that
 $$ P \,(\hypc_0 (z) = 1) \ = \ P \,(\spar_0 (x) = 1 \ \hbox{for all} \ x \in H_z) \ = \ p^{2^d} \ = \ q \ = \ P \,(\boot_0 (z) = 1), $$
 assuming that~$p > 0$ and applying Lemmas~\ref{lem:attractive}~and~\ref{lem:coupling}, we get
 $$ \begin{array}{rcl}
    \lim_{t \to \infty} \,P \,(\best_t (x) = 1) & \geq & \lim_{t \to \infty} \,P \,(\spar_t (x) = 1) \ = \ P \,(x \in \spar_{\infty}) \vspace*{4pt} \\
                                                & \geq &  P \,(\spar_{\infty} = \Z^d) \ = \ P \,(\hypc_{\infty} = \Z^d) \ \geq \ P \,(\boot_{\infty} = \Z^d) \ = \ 1 \end{array} $$
 which proves the lemma as well as Theorem~\ref{th:selfish}.
\end{proof}} \\


\noindent\textbf{Acknowledgment}.
 The authors would like to thank Rick Durrett and an anonymous referee who independently underlined the connection between our model
 and bootstrap percolation, which has considerably strengthened our results, as well as another referee for useful comments.


\end{document}